\documentclass[10pt, reqno]{amsart}
\usepackage{graphicx, amssymb, amsmath, amsthm, hyperref}
\numberwithin{equation}{section}

\allowdisplaybreaks

\let\Re=\undefined\DeclareMathOperator*{\Re}{Re}

\DeclareMathOperator*{\spaan}{span}

\newcommand{\R}{\mathbb{R}}
\newcommand{\C}{\mathbb{C}}

\newcommand{\M}{\mathcal{M}}

\newtheorem{theorem}{Theorem}[section]

\newtheorem*{theorem*}{Theorem}

\newtheorem{lemma}[theorem]{Lemma}

\newtheorem{corollary}[theorem]{Corollary}

\newtheorem{proposition}[theorem]{Proposition}

\theoremstyle{definition}
\newtheorem{definition}[theorem]{Definition}

\theoremstyle{remark}

\newcommand{\qtq}[1]{\quad\text{#1}\quad}
\let\HH=\H
\renewcommand{\H}{\mathcal{H}}

\begin{document}

\title[Determination of the nonlinearity]{Determination of Schr\"odinger nonlinearities \\ from the scattering map}

\author[R. Killip]{Rowan Killip}
\email{killip@math.ucla.edu}
\address{Department of Mathematics, University of California Los Angeles}

\author[J. Murphy]{Jason Murphy}
\email{jamu@uoregon.edu}
\address{Department of Mathematics, University of Oregon}

\author[M. Visan]{Monica Visan}
\email{visan@math.ucla.edu}
\address{Department of Mathematics, University of California Los Angeles} 

\maketitle

\begin{abstract} We prove that the small-data scattering map uniquely determines the nonlinearity for a wide class of gauge-invariant, intercritical nonlinear Schr\"odinger equations.  We use the Born approximation to reduce the analysis to a deconvolution problem involving the distribution function for linear Schr\"odinger solutions.  We then solve this deconvolution problem using the Beurling--Lax Theorem.
\end{abstract}

\section{Introduction}

We consider nonlinear Schr\"odinger equations (NLS) of the form
\begin{equation}\label{nls}
(i\partial_t + \Delta) u = F(t,x,u),
\end{equation} 
where $(t,x)\in\R\times\R^d$ with $d\geq 1$ and $u:\R\times\R^d\to\C$.  We treat a general class of nonlinearities for which the equation \eqref{nls} admits a `small-data scattering theory' in $H^1(\R^d)$.  By this, we mean that for any sufficiently small $u_-\in H^1$, there exists a unique, global-in-time solution $u$ to \eqref{nls} and a $u_+\in H^1$ such that
\[
\lim_{t\to\pm\infty}\|u(t)-e^{it\Delta}u_\pm\|_{H^1}=0,
\]
where $e^{it\Delta}$ denotes the linear Schr\"odinger propagator; see Theorem~\ref{T:scatter} for more details.  To derive small-data scattering in $H^1$ for \eqref{nls}, it suffices assume that $F(t,x,u)$ decays rapidly enough as $|u|\to0$ and has controlled growth as $|u|\to\infty$; see Definition~\ref{D} below for the specific class of nonlinearities we consider.  

The small-data scattering theory for \eqref{nls} allows us to define the \emph{scattering map} $S_F$, which sends the asymptotic state $u_-$ at $t=-\infty$ to the asymptotic state $u_+$ at $t=+\infty$.  The question we would like to answer in this paper is the following inverse problem: Does the scattering map $S_F$ uniquely determinez the nonlinearity?  Our main result (Theorem~\ref{T} below) answers this question in the affirmative for a general class of nonlinearities. 

The precise assumptions we make on the nonlinear term $F$ are as follows:

\begin{definition}\label{D} Let $F:\R\times\R^d\times\C\to\C$.  We call a continuous function $F$ \emph{admissible with parameters $(p_0,p_1)$} if $F(t,x,u)=\rho(t,x,|u|^2)u$ for some real-valued function $\rho:\R\times\R^d\times[0,\infty)\to \R$ satisfying the following:
\begin{align*}
\partial_x^\alpha \rho(t,x,0)\equiv 0 \qtq{for}|\alpha|\leq 1,
\end{align*}
and there exists $C>0$ such that
\begin{align}\label{ad B}
\sup_{(t,x)\in\R\times\R^d} |\partial_x^\alpha \partial_\lambda \rho (t,x,\lambda)| \leq C\sum_{p\in\{p_0,p_1\}} \lambda^{\frac{p}{2}-1} \qtq{for}|\alpha|\leq 1,
\end{align}
where $p_0=\tfrac{4}{d}$ and
\[
p_0<p_1 = \begin{cases} \text{arbitrarily large but finite,} & \text{ if } d\in\{1,2\}, \\ \tfrac{4}{d-2}, &  \text{ if }  d\geq 3.\end{cases}
\]
\end{definition}

For $d=1,2$, the classes of admissible nonlinearities are nested: If \eqref{ad B} holds for some pair $(p_0,p_1)$, then it also holds with $(p_0,p_2)$ when $p_2\geq p_1$.  

The hypotheses on the parameters $(p_0,p_1)$ ensure that  \eqref{nls} is `{intercritical}'.  Indeed, the standard power-type NLS 
\begin{equation}\label{standard-NLS}
(i\partial_t + \Delta)u = \pm|u|^p u
\end{equation}
is invariant under the rescaling $u(t,x)\mapsto\lambda^{\frac{2}{p}}u(\lambda^2 t,\lambda x)$.  The homogeneous $L^2$-based Sobolev space of initial data that is invariant under this rescaling is $\dot H^{s(p)}(\R^d)$, where $s(p):=\tfrac{d}{2}-\tfrac{2}{p}$.  A well-established argument (employing Strichartz estimates and contraction mapping) guarantees local well-posedness of \eqref{standard-NLS} in $H^s(\R^d)$ for $s\geq s(p)$. The special cases $s(p)=0$ and $s(p)=1$ correspond to $p=\tfrac{4}{d}$ and $p=\tfrac{4}{d-2}$ (with $d\geq 3$) and are known as the mass- and energy-critical problems, respectively.  When $s(p)\in[0,1]$ we call the equation \eqref{standard-NLS} `intercritical'. 

The definition of admissibility ensures that the nonlinearities we consider satisfy the bounds
\[
|F(t,x,u)|\lesssim |u|^{p_0+1} + |u|^{p_1+1} ,
\]
with $s(p_0)=0$ and $s(p_1)\leq 1$, so that the equation \eqref{nls} may also be described as intercritical.  

We have several reasons for restricting our attention to `intercritical' nonlinearities: First, as $s(p_0)=0$, the exponent $p_0=\tfrac{4}{d}$ is the lowest we can include and still obtain scattering in the standard Sobolev space framework (without introducing weights, for example). The restriction $s(p_1)\leq 1$ is a simple way to guarantee that we never need to differentiate the nonlinearity more than once.  This simplifies the technical aspects of the small-data scattering analysis, thus allowing us to focus on the main ideas involved in the recovery of the nonlinearity.

In Theorem~\ref{T:scatter} below, we prove a small-data scattering theory for nonlinear Schr\"odinger equations of the form \eqref{nls} with admissible nonlinearities.  Given admissible nonlinearities $F_j$ with parameters $(p_0,p_j)$, we show that we can define the small-data scattering maps $S_j$ on sufficiently small balls $B_j$ in $H^{s(p_j)}$, where  $s(p_j):=\tfrac{d}{2}-\tfrac{2}{p_j}$ (see Definition~\ref{D:S} below). Note that the intersection $B_1\cap B_2$ is a neighborhood of zero in $H^{s(\max\{p_1,p_2\})}$, so that there is a common domain on which we can compare the scattering behaviors.

Our main result is the following theorem:

\begin{theorem}[The scattering map determines the nonlinearity]\label{T} Let $F_1$ and $F_2$ be admissible nonlinearities in the sense of Definition~\ref{D}.  If the corresponding scattering maps, $S_j:B_j\to H^{s(p_j)}$, satisfy $S_1=S_2$ on $B_1\cap B_2$, then $F_1\equiv F_2$. 
\end{theorem}

The problem of recovering unknown parameters (including external potentials, as well as nonlinearities) from the scattering data is a classical problem that has received significant interest in the setting of nonlinear dispersive PDE.  In what follows, we will review the literature that is most closely related to our main result, specifically focusing on `time-dependent' scattering problems. 

Many previous works on recovering the nonlinearity from the scattering map rely on fairly strong assumptions on the nonlinearity.  This includes assumptions such as analyticity, along with structural assumptions in which one assumes the nonlinearity has a certain form (e.g. $F(x,u)=\alpha(x)|u|^p u$ or $F(x,u)=(|x|^{-\gamma}\ast|u|^2)u$) and seeks to recover unknown parameters (i.e. $p$ and $\alpha$ in the first example, or $\gamma$ in the second).  We refer the reader to \cite{CarlesGallagher, PauStr} for treatments of the analytic case, \cite{ChenMurphy1, ChenMurphy2,  EW, MorStr, Murphy, Strauss, Watanabe, Weder0, Weder1, Weder6, Weder3, Weder4, Weder5} for treatments of power-type and related cases, and \cite{Sasaki2, Sasaki, Sasaki3, SasakiWatanabe, Watanabe0} for treatments of Hartree-type cases.  See also  \cite{SBS2, HMG, LeeYu} for related work.

We were inspired to consider the problem of recovering unknown nonlinearities from the scattering map by the work \cite{SBUW}, which considered this problem in the setting of quintic-type nonlinear wave equations and studied the problem using microlocal analysis techniques.  In particular, in our previous work \cite{KMV} we proved a result similar to Theorem~\ref{T}, introducing a new technique that reduces the analysis to a deconvolution problem. In \cite{KMV}, we only treated the two-dimensional NLS and did not consider nonlinearities that may also depend on the time and space variables.  The techniques introduced in \cite{KMV} were subsequently extended to the setting of the nonlinear wave equation in \cite{HKV}, which strengthened the original results of \cite{SBUW} in several directions.

The role of this paper is to further develop the techniques introduced in \cite{KMV}, thereby expanding their applicability.   In particular, in this work we remove the restriction on spatial dimension and broaden the class of nonlinearities under consideration, by allowing dependence on the space and time variables.

For the remainder of the introduction, we outline the strategy of the proof of Theorem~\ref{T}.  Using the Duhamel formula, one finds that the scattering map satisfies the following implicit formula:
\begin{equation}\label{implicit-formula}
S_F(u_-)=u_--i\int_\R e^{-it\Delta}F(t,x,u(t))\,dt,
\end{equation}
where $u$ is the solution to \eqref{nls} that scatters to $u_-$ as $t\to-\infty$.  Replacing the full solution $u(t)$ with $e^{it\Delta}u_-$ on the right-hand side of \eqref{implicit-formula} constitutes the \emph{Born approximation},
\[
 S_F^B : \varphi\mapsto \varphi-i\int_\R e^{-it\Delta}F(t,x,e^{it\Delta}\varphi)\,dt,
\]
to the scattering map.  As we will see, the Born approximation accurately describes the small-data regime.  In this way, we will show that knowledge of the scattering map completely determines integrals of the form
\[
i \bigl\langle \varphi, [ I - S_F^B ](\varphi)\bigr\rangle =  
\iint G(t,x,|e^{it\Delta}\varphi|^2)\,dt\,dx,
\]
where $G(t,x,|z|^2)=\bar z F(t,x,z)$ and $\varphi\in H^1$; see Lemma~\ref{L:Born}.

Next, by choosing linear solutions that concentrate around a single point $(t_0,x_0)$ in space-time, we will see that knowledge of the scattering map allows one to effectively evaluate integrals of the form
\begin{equation}\label{PE}
\iint G(t_0,x_0,|e^{it\Delta}\varphi|^2)\,dt\,dx
\end{equation}
for fixed $(t_0,x_0)\in\R\times\R^d$; see Lemma~\ref{L:localize}.

Using the Fundamental Theorem of Calculus and Fubini's Theorem, integrals of the form \eqref{PE} may be rewritten in terms of the distribution function $\mu$ of the function $(t,x)\mapsto |e^{it\Delta}\varphi|^2(x)$.  Varying the amplitude of the data, one can recognize the resulting integral as a type of convolution of the nonlinearity with a fixed weight.  Thus, the original problem reduces to one of deconvolution.  This reduction is carried out in detail in Section~\ref{S:Reduction}.

The final ingredient in the proof entails specializing to Gaussian data (for which the corresponding linear solutions remain Gaussian for all time).  In this case, we can derive sufficient information about the convolution weight to successfully resolve the deconvolution problem.  Precisely, this requires that we prove that the Laplace transform of the function $k\mapsto \mu(e^{-k})$ is an outer function on a suitable half plane.  With this input, we can use the Beurling--Lax Theorem of analytic function theory to solve the deconvolution problem.  This is accomplished in Section~\ref{S:Deconvolution}.

The rest of this paper is organized as follows: Section~\ref{S:prelim} collects the technical preliminaries needed in the remainder of the paper. This includes Section~\ref{S:BL}, which provides an introduction to Hardy spaces and the Beurling--Lax Theorem. Section~\ref{S:scatter} contains the proof of the small-data scattering theory for \eqref{nls} with admissible nonlinearities (see Theorem~\ref{T:scatter}). Section~\ref{S:Reduction} reduces the proof of the main result (Theorem~\ref{T}) to a deconvolution problem (see Proposition~\ref{P:agree} and Corollary~\ref{C:conv0}). Finally, Section~\ref{S:Deconvolution} resolves this deconvolution problem and includes the proof of the main result, Theorem~\ref{T}.

\subsection*{Acknowledgements} R.K. was supported by NSF grant DMS-2154022;  J.M. was supported by NSF grant DMS-2350225; M.V. was supported by NSF grant DMS-2054194.

\section{Preliminaries}\label{S:prelim}

We write $A\lesssim B$ or $A=\mathcal{O}(B)$ to denote $A\leq CB$ for some absolute constant $C>0$.  We will use $A\approx B$ to denote $A\lesssim B\lesssim A$.  We write $f(\sigma)= o(\sigma^C)$ when $\sigma^{-C}f(\sigma)\to 0$ as $\sigma \to 0$. We indicate dependencies on additional parameters via subscripts.  

We write $L_t^q L_x^r(I\times\R^d)$ for the mixed Lebesgue space on a space-time slab $I\times\R^d$, equipped with the norm
$$
\|u\|_{L_t^q L_x^r(I\times\R^d)} = \bigl\| \|u(t)\|_{L_x^r(\R^d)} \bigr\|_{L_t^q(I)}.
$$
We use $H^{s,r}$ to denote the inhomogeneous Sobolev space with norm
\[
\|u\|_{H^{s,r}} = \|u\|_{L^r} + \| |\nabla|^s u\|_{L^r},
\]
and we denote the $L^2$ inner product by $\langle\cdot,\cdot\rangle$. 

We will make use of the standard Strichartz estimates for the linear Schr\"odinger equation.  We call a pair $(q,r)\in[2,\infty]\times[2,\infty]$ \emph{Schr\"odinger admissible} in $d$ dimensions if $\tfrac{2}{q}+\tfrac{d}{r}=\tfrac{d}{2}$ and $(d,q,r)\neq(2,2,\infty)$.  Correspondingly, we call a space $L^q_tL_x^r$ Schr\"odinger admissible if the pair $(q,r)$ is Schr\"odinger admissible. 

\begin{lemma}[Strichartz estimates, \cite{GinibreVelo, KeelTao, Strichartz}] For any Schr\"odinger admissible pair $(q,r)$ and any $\varphi\in L^2(\R^d)$, 
\[
\|e^{it\Delta}\varphi\|_{L_t^q L_x^r(\R\times\R^d)}\lesssim \|\varphi\|_{L^2}. 
\]
Given an interval $I$ containing $0$, Schr\"odinger admissible pairs $(q,r),(\tilde q,\tilde r)$, and $F\in L_t^{\tilde q'}L_x^{\tilde r'}(I\times\R^d)$, we have
\[
\biggl\| \int_0^t e^{i(t-s)\Delta}F(s)\,ds \biggr\|_{L_t^q L_x^r(I\times\R^d)}\lesssim \|F\|_{L_t^{\tilde q'}L_x^{\tilde r'}(I\times\R^d)}. 
\]
\end{lemma}

\subsection{Choice of function spaces} 
We will employ the spaces $Y,Y'$ defined by 
\begin{equation}\label{spaces}
Y = L_{t,x}^{\frac{2(d+2)}{d}}\quad\text{and}\quad  Y' = L_{t,x}^{\frac{2(d+2)}{d+4}}. 
\end{equation}

We write $s(L_t^q L_x^r) = \tfrac{d}{2}-[\tfrac{2}{q}+\tfrac{d}{r}]$ for the Sobolev regularity associated to the Lebesgue space $L_t^q L_x^r$ under the Schr\"odinger scaling and $s(p)=\tfrac{d}{2}-\tfrac{2}{p}$ for the critical regularity associated to the power-type nonlinearity $|u|^p u$.  Note that $Y$ is Schr\"odinger admissible and so $s(Y)=0$.  

In what follows, we will choose $p\in\{p_0,p_1\}$, where $(p_0,p_1)$ are the parameters of an admissible nonlinearity (see Definition~\ref{D}). We further define the spaces
\begin{equation}\label{spaces'}
X_p = L_{t,x}^{\frac{p(d+2)}{2}} \quad\text{and}\quad \bar X_p = L_t^{\frac{p(d+2)}{2}} L_x^{\frac{2dp(d+2)}{dp(d+2)-8}}.
\end{equation}
Note that $\bar X_p$ is  Schr\"odinger admissible and so $s(\bar X_p)=0$. By Sobolev embedding,
\begin{align}\label{embedding}
\| u\|_{X_p} \lesssim \| |\nabla|^{s(p)} u\|_{\bar X_p}. 
\end{align}
Moreover, we have the following H\"older estimate
\begin{align*}
\| |u|^p u\|_{Y'} \lesssim \|u\|_{X_p}^p \|u\|_Y. 
\end{align*}

In dimensions $d\in\{1,2\}$, all power-type nonlinearities are energy-subcritical; in particular, $s(p_1)<1$. Correspondingly, we will need the following fractional chain rule estimate to establish the small-data scattering theory.  If $F(t,x,u)$ were independent of $x$, then the classical results described in \cite{CW, Taylor} would suffice.

\begin{proposition}[Fractional chain rule]\label{P:FCR} Let $F:\R^d\times\C\to\C$. Suppose that
\[
\partial_x^\alpha F(x,0) = 0 \qtq{for all}x\in\R^d \qtq{and} |\alpha| \leq 1
\]
and that there exists $L:\C\to[0,\infty)$ such that
\[
\sup_{x\in\R^d} |\partial_x^\alpha F(x,u)-\partial_x^\alpha F(x,v)| \leq [L(u)+L(v)]|u-v|
\]
for all $u,v\in\C$ and $|\alpha|\leq 1$. 

Then for any $s\in(0,1]$, $r,r_1\in(1,\infty)$, and $r_2\in(1,\infty]$ satisfying $\tfrac{1}{r}=\tfrac{1}{r_1}+\tfrac{1}{r_2}$, we have
\[
\|F(x,u)\|_{H^{s,r}} \lesssim \|L(u)\|_{L^{r_2}} \| u\|_{H^{s,r_1}}
\]
for any $u:\R^d\to\C$. 
\end{proposition}

\begin{proof} The case $s=1$ follows from H\"older's inequality and the standard chain rule; thus, it suffices to consider $s\in(0,1)$. 

By the Littlewood--Paley square function estimate, we may bound
\begin{equation}\label{LPsf}
\| F(x,u)\|_{H^{s,r}} \lesssim  \|F(x,u)\|_{L^r}+ \biggl\| \biggl(\sum_{N\geq 1} N^{2s} |P_NF(x,u)|^2\biggr)^{\frac12}\biggr\|_{L^r},
\end{equation}
where the sum is over dyadic numbers and $P_N$ is the standard Littlewood--Paley projection onto frequencies $|\xi|\approx N$. 

By H\"older's inequality and the assumptions on $F$, we observe that
\[
\|F(x,u)\|_{L^r} \lesssim \|L(u)\|_{L^{r_2}} \|u\|_{L^{r_1}},
\]
which is acceptable. 

Writing $\check\psi$ for the convolution kernel of $P_1$, we use the fact that $\int \check\psi = 0$ to obtain
\[
 P_N F(x,u(x)) = \int N^d \check\psi(Ny)[F(x-y,u(x-y))-F(x,u(x))]\,dy. 
\]
We now write
\begin{align}
F(x-y,u(x-y))&-F(x,u(x)) \nonumber \\
& = F(x-y,u(x-y))-F(x,u(x-y)) \label{1123-1} \\
& \quad + F(x,u(x-y)) - F(x,u(x)). \label{1123-2}
\end{align}

We first observe that by the Fundamental Theorem of Calculus and the properties of $F$, we have
\begin{align*}
|\eqref{1123-1}|&  = \biggl| \int_0^1 [y \cdot \nabla_x F] (x-\theta y,u(x-y)) \,d\theta \biggr|  \lesssim |y|\,L(u(x-y))|u(x-y)|. 
\end{align*}
In this way, we obtain
\begin{align*}
\int N^d |\check\psi(Ny)|\, | \eqref{1123-1}|\,dy 
& \lesssim N^{-1} \int N^d |Ny| |\check \psi(Ny)| \, L(u(x-y)) \, |u(x-y)|\,dy \\
& \lesssim N^{-1}\M[L(u)u](x), 
\end{align*}
where $\M$ is the Hardy--Littlewood maximal function. Thus, using the maximal function estimate and H\"older's inequality, the contribution of \eqref{1123-1} to the sum in \eqref{LPsf} can be bounded by
\begin{align*}
\biggl\| \biggl(\sum_{N\geq 1} N^{2s-2}\biggr)^{\frac12}\M[L(u)u]\biggr\|_{L^r} \lesssim \|L(u)\|_{L^{r_2}}\|u\|_{L^{r_1}},
\end{align*}
which is acceptable. 

Next, we use the properties of $F$ to bound
\[
|\eqref{1123-2}| \lesssim [L(u(x-y))+L(u(x))]|u(x-y)-u(x)|. 
\]
The contribution of this term can now be estimated exactly as in the proof of the standard fractional chain rule; see for example \cite[Proposition~5.1]{Taylor}.  In particular, the contribution of \eqref{1123-2} to the sum in \eqref{LPsf} is bounded by
\[
\| L(u)\|_{L^{r_2}} \||\nabla|^s u\|_{L^{r_1}},
\]
which is acceptable.\end{proof}

\subsection{The Beurling--Lax Theorem}\label{S:BL} In this subsection, we are interested in pairs of functions $v,\varphi \in  L^2([0,\infty))$ that satisfy
\begin{equation}\label{BL question}
\int_0^\infty \overline{\varphi(k+\ell)} v(k) \,dk = 0 \qtq{for all} \ell\in[0,\infty).
\end{equation}
The specific question that we discuss is this: For which functions $v$ does \eqref{BL question} imply that $\varphi=0$?  As we will see, the Beurling--Lax Theorem provides a complete solution to this problem in terms of the Laplace transform of $v$.  For further details on what follows, we recommend the textbooks \cite{Garnett,Hoffman}. 

For $v,\varphi \in L^2([0,\infty))$, 
\begin{equation}\label{v to V}
V(z) := \int_0^\infty e^{-kz} v(k) \,dk \qtq{and} \Phi(z):=\int_0^\infty e^{-kz} \varphi(k) \,dk
\end{equation}
define analytic functions in the right half-plane $\mathbb{H}:=\{z\in\C:\Re z>0\}$.  Moreover,
\begin{equation}\label{Hardy cond}
 \sup_{x>0} \int_\R \bigl| V(x+iy)\bigr|^2\,dy < \infty .
\end{equation}
The same holds for $\Phi(z)$, of course, but let us focus on $V(z)$ for the moment.

The space of functions $V(z)$ analytic in $\mathbb{H}$ and satisfying \eqref{Hardy cond} is known as the Hardy space $\H^2(\mathbb H)$.  By the Paley--Wiener Theorem, this space is precisely the image of $L^2([0,\infty))$ under the mapping \eqref{v to V}.

The standard tools of harmonic analysis guarantee that the limit
\begin{equation}\label{V BVs}
V(iy) := \lim_{x\downarrow 0} V(x+iy)
\end{equation}
exists in both $L^2(\R)$ and a.e. senses.  Moreover, one may recover $V(z)$ from its boundary values via the Poisson integral formula.  Boundary values also provide the simplest definition of the inner product on the Hilbert space $\H^2(\mathbb H)$:
\begin{equation}\label{2.11}
\langle \Phi, V \rangle_{\H^2} := \int_{-\infty}^\infty \overline{\Phi(iy)} V(iy) \, dy = 2\pi \int_0^\infty \overline{\varphi(k)} v(k) \, dk. 
\end{equation}
The last equality here follows from the Plancherel identity.

Except in the case $V\equiv 0$, the boundary values of a function $V\in \H^2(\mathbb H)$ cannot vanish on a set of positive measure; indeed, Szeg\HH{o} proved that
\begin{equation}\label{BVSzego}
\int \bigl|\log |V(iy)| \bigr| \frac{dy}{1+y^2} < \infty.
\end{equation}
In view of this, if $V\in \H^2$ and $V\not\equiv0$, we may define an analytic function on $\mathbb H$ by
\begin{equation}\label{OV defn}
 O_V(z) := \exp\biggl\{ \int_\R \log |V(it)| \Bigl[\frac{1}{z-it} - \frac{it}{1+t^2} \Bigr] \frac{dt}{\pi} \Bigr\} .
\end{equation}
This construction ensures that $\log|O_V(z)|$ is the Poisson integral of the boundary values $\log|V(iy)|$.  In particular, $\log|O_V(z)|$ is harmonic. In general, $\log|V(z)|$ is only subharmonic; for example, $V$ may have zeros in the half-plane $\mathbb H$.  This thinking leads us naturally to two important discoveries of Riesz: For all $0\not\equiv V\in \H^2(\mathbb H)$,
\begin{equation}\label{Riesz1}
 O_V(z) \in \H^2(\mathbb H) \qtq{and} |V(z)|  \leq |O_V(z)|  \quad\text{for all $z\in\mathbb H$.}
\end{equation}

If $|V(z)| = |O_V(z)|$ for a single $z\in\mathbb H$, then this holds for all $z\in\mathbb H$, by the strong maximum principle.  Functions $V(z)$ for which this holds are known as \emph{outer functions}.  The function $O_V(z)$ is an example of an outer function.

From \eqref{Riesz1}, we see that the analytic function $I_V(z):=V(z)/O_V(z)$ satisfies $|I_V(z)|\leq 1$ throughout $\mathbb H$.  Moreover, by construction, the boundary values (which exist a.e.) satisfy $|I_V(iy)|=1$.  An analytic function with these two properties is termed an \emph{inner function}.

Evidently, $V(z)=I_V(z) O_V(z) $.  This constitutes an \emph{inner/outer factorization} of $V(z)$.  Such a factorization is unique up to the multiplication of each factor by (complementary) unimodular complex numbers.

We require just one more preliminary before we can state the Beurling--Lax Theorem.  A vector-subspace $\mathcal S$ of $\H(\mathbb H)$ is called \emph{shift invariant} if
\begin{equation}\label{shift inv}
 V(z)\in\mathcal S \implies e^{-z\ell}V(z)\in\mathcal S \text{ for all $\ell\geq 0$.}
\end{equation}
This name becomes more reasonable when we see how shifting a function $v\in L^2([0,\infty))$ to the right by an amount $\ell\geq 0$ affects its Laplace transform $V(z)$:
\begin{equation}\label{shifted LT}
\int_0^\infty e^{-zk} [\chi_{[0,\infty)} \cdot v](k-\ell) \,dk = e^{-z\ell} \int_0^\infty e^{-zk} v(k)\,dk = e^{-z\ell} V(z).
\end{equation}

\begin{theorem}[Beurling--Lax]\label{T:BL} Any nonzero, closed, shift-invariant subspace $\mathcal S$ of $\H^2(\mathbb H)$ is of the form $\mathcal S=J\H^2(\mathbb H)$ for some inner function $J(z)$. 
\end{theorem} 

This is the form of the theorem presented (and proved) in \cite{Hoffman}.  Historically, Beurling \cite{Beurling} proved the analogue of this theorem for analytic functions on the unit disk with multiplication by $z^n$, $n\in\mathbb N$, which corresponds to a shift of the Taylor coefficients.  The half-plane form stated above was subsequently proved by Lax \cite{Lax}.  A proof of the Beurling formulation can be found in both \cite{Garnett} and \cite{Hoffman}.  In \cite{Hoffman}, it is also shown how one may deduce each version of the result from the other.

We may now demonstrate how the Beurling--Lax Theorem solves the problem stated at the beginning of this subsection.   

\begin{corollary}\label{C:BL} Suppose $v\in L^2([0,\infty))$ and $V(z)$, defined by \eqref{v to V}, is outer. If $\varphi\in L^2([0,\infty))$ satisfies \eqref{BL question}, then $\varphi\equiv 0$.

Conversely, if $v\in L^2([0,\infty))$ and $V(z)$ is not outer, then there is a non-zero $\varphi\in L^2([0,\infty))$ so that \eqref{BL question} holds.
\end{corollary}

\begin{proof}
Viewed through the lens of \eqref{2.11} and \eqref{shifted LT}, we see that \eqref{BL question} becomes
\begin{equation}\label{2.17}
\langle \Phi(z), e^{-z\ell} V(z) \rangle_{\H^2} =0 \quad\text{for all $\ell\geq 0$}.
\end{equation}
This is equivalent to saying that $\Phi$ is orthogonal to
\begin{equation}\label{2.18}
\mathcal S_V:= \overline{\spaan\{ e^{-\ell z} V(z) : \ell\geq 0\}},
\end{equation}
where the closure is taken in $\H^2(\mathbb H)$.  In this way, the corollary is reduced to the following assertion:
\begin{equation}\label{2.19}
\text{$V(z)$ is outer}   \iff   \mathcal S_V=\H^2(\mathbb H).
\end{equation}
We note that $\mathcal S_V=\{0\}$ if and only if $V\equiv 0$ (which is not outer) and so may exclude these cases from further consideration. 

As  $\mathcal S_V$ is shift invariant, it admits the representation $\mathcal S_V=J(z)\H^2(\mathbb H)$ for some inner function $J(z)$. In particular, $V(z)$ admits the representation $V(z)=J(z)W(z)$ for some $W(z)\in\H^2(\mathbb H)$.  Factoring $W(z)$ yields
\begin{equation}\label{2.20}
V(z)=J(z)I_W(z)O_W(z).
\end{equation}
This constitutes an inner/outer factorization of $V(z)$; the inner factor is $J(z)I_W(z)$ and the outer factor is $O_W(z)$.

Suppose now that $V(z)$ is outer.  By uniqueness of the factorization, it follows that $J(z)I_W(z)$ is a unimodular constant and consequently,
$$
\mathcal S_V= J(z)\H^2(\mathbb H)  \supseteq J(z)I_W(z)\H^2(\mathbb H) = \H^2(\mathbb H).
$$
Thus, when $V(z)$ is outer, $\mathcal S_V=\H^2(\mathbb H)$.

To prove the converse, we now suppose that $V(z)=I_V(z) O_V(z)$ is not outer, which is to say, $I_V(z)$ is not a unimodular constant.  Evidently, $e^{-\ell z}V(z)$ belongs to $I_V\H^2(\mathbb H)$ for all $\ell\geq 0$.  As this space is $\H^2$-closed, it follows that $\mathcal S_V\subseteq I_V\H^2(\mathbb H)$.  The uniqueness (modulo unimodular constants) of the inner/outer factorization together with the maximum modulus principle then shows that $\mathcal S_V$ contains no outer functions and consequently, $\mathcal S_V\neq \H^2(\mathbb H)$.
\end{proof}

In order to prove Theorem~\ref{T}, we will need to show that the convolution equation \eqref{BL question} has a unique solution for a very specific choice of $v(k)$.  Although we are able to compute the Laplace transform $V(z)$ of this function, the expression is quite complicated.  With this in mind, it is convenient to have a simple direct criterion for demonstrating that $V(z)$ is outer:

\begin{proposition}\label{P:outer}
Suppose $V(z)\in \H^2(\mathbb H)$ and $W(z):=(1+z)^{-n} / V(z)\in \H^2(\mathbb H)$ for some $n\in \mathbb N$.  Then $V(z)$ is an outer function.
\end{proposition}

\begin{proof}
We begin by showing that $z\mapsto (1+z)^{-n}$ is an outer function.  As mentioned earlier, this means verifying that $-n\log|1+z|$ is equal to the Poisson integral of its boundary values at a least one point $z\in\mathbb H$.  This is easily checked at $z=1$, noting that
$$
\int_{-\infty}^\infty  \frac{\log| 1 + i y|}{\pi(1+y^2)}\,dy= \log(2), 
$$
as may be verified, for example, by the Cauchy integral formula.
  
Performing an inner/outer factorization of both $V(z)$ and $W(z)$, we find
\begin{equation*}
(1+z)^{-n} = V(z)W(z) = I_V(z)I_W(z) \cdot O_V(z)O_W(z).
\end{equation*}
By uniqueness of the inner/outer factorization, we deduce that $I_V(z)I_W(z)$ is a unimodular constant.  By the maximum modulus principle, this implies that both $I_V(z)$ and $I_W(z)$ are unimodular constants. Thus, $V(z)$ (and also $W(z)$) is an outer function.\end{proof}

\section{Small-data scattering}\label{S:scatter}

In this section, we will establish a small-data scattering theory for nonlinear Schr\"odinger equations of the form \eqref{nls} with admissible nonlinearities (in the sense of Definition~\ref{D}).  We will employ the function spaces introduced in \eqref{spaces} and \eqref{spaces'}, as well as the notation $s(p)=\tfrac{d}{2}-\tfrac{2}{p}$.

\begin{theorem}[Small data scattering]\label{T:scatter} Let $F$ be admissible with parameters $(p_0,p_1)$ and define
\[
B_\eta = \{f\in H^{s(p_1)}:\|f\|_{H^{s(p_1)}}<\eta\},\quad \eta>0.
\]

There exists $\eta>0$ sufficiently small so that for any $u_-\in B_\eta$, there exists a unique global solution $u:\R\times\R^d\to\C$ to \eqref{nls} with
\[
\lim_{t\to-\infty}\|u(t)-e^{it\Delta}u_-\|_{H^{s(p_1)}}=0.
\]
This solution satisfies the global space-time bounds
\[
\| |\nabla|^{s(p)}u\|_{\bar X_p\cap Y} \lesssim \|u_-\|_{H^{s(p)}},\quad p\in\{p_0,p_1\}
\]
and scatters to a unique $u_+\in H^{s(p_1)}$ as $t\to\infty$, that is,
\[
\lim_{t\to\infty} \|u(t)-e^{it\Delta}u_+\|_{H^{s(p_1)}}=0.
\]
\end{theorem}

Using Theorem~\ref{T:scatter}, we can define the small-data scattering map for an admissible nonlinearity $F$:

\begin{definition}\label{D:S} Under the hypotheses of Theorem~\ref{T:scatter}, we define the \emph{scattering map} $S_F:B_\eta\to H^{s(p_1)}$ by $S_F(u_-)=u_+$. 
\end{definition}

\begin{proof}[Proof of Theorem~\ref{T:scatter}] Define the map
\[
\Phi(u) = e^{it\Delta}u_- - i\int_{-\infty}^t e^{i(t-s)\Delta}F(s,x,u(s))\,ds. 
\]
We let $(Z,d)$ be the complete metric space given by
\begin{align*}
Z = \bigl\{u:\R\times\R^d\to\C: & \||\nabla|^{s(p)}u\|_{\bar X_p\cap Y} \leq 4C\|u_-\|_{H^{s(p)}}\text{ for each } p\in\{p_0,p_1\}\bigr\}
\end{align*}
and distance function
\[
d(u,v) = \|u-v\|_Y. 
\]
Here $C>0$ is a universal constant dictated by the implicit constants appearing in Strichartz estimates, Sobolev embedding, and the fractional chain rule. Throughout the proof, space-time norms are taken over $\R\times\R^d$ unless otherwise indicated. 

We first prove that for $\eta$ sufficiently small, $\Phi:Z\to Z$.  Given $u\in Z$, Strichartz estimates, the assumptions on $F$, and H\"older's inequality show
\begin{align*}
\|\Phi(u)\|_{\bar X_{p_0}\cap Y} & \lesssim \|u_-\|_{L^2} + \|F(t,x,u)\|_{Y'} \\
& \lesssim \|u_-\|_{L^2}+\sum_{p\in\{p_0,p_1\}} \| |u|^p u\|_{Y'} \\
& \lesssim \|u_-\|_{L^2} + \sum_p \|u\|_{X_p}^p \|u\|_Y.
\end{align*}
Using \eqref{embedding} and the fact that $u\in Z$, we therefore obtain
\begin{align*} 
\|\Phi(u)\|_{\bar X_{p_0}\cap Y}  & \lesssim \|u_-\|_{L^2} + \sum_p \| |\nabla|^{s(p)} u\|_{\bar X_p}^p \|u_-\|_{L^2} \\
&\lesssim \|u_-\|_{L^2} + \sum_p \|u_-\|_{H^{s(p)}}^p\|u_-\|_{L^2} \\
& \lesssim \|u_-\|_{L^2} + [\eta^{p_0}+\eta^{p_1}]\|u_-\|_{L^2} \leq 4C\|u_-\|_{L^2}
\end{align*}
for $\eta$ sufficiently small. Next, we use the fractional chain rule (Proposition~\ref{P:FCR}), the assumptions on $F$, H\"older, and \eqref{embedding} to obtain
\begin{align*}
\| |\nabla&|^{s(p_1)} \Phi(u)\|_{\bar X_{p_1} \cap Y} \\
& \lesssim \| |\nabla|^{s(p_1)}u_-\|_{L^2} + \||\nabla|^{s(p_1)} F(t,x,u)\|_{Y'} \\
& \lesssim \|u_-\|_{H^{s(p_1)}} + \biggl\| \biggl\|\sum_{p\in\{p_0,p_1\}} |u|^p\biggr\|_{L_x^{\frac{d+2}{2}}} \| u\|_{H^{s(p_1),\frac{2(d+2)}{d}}}\biggr\|_{L_t^{\frac{2(d+2)}{d+4}}} \\
& \lesssim \|u_-\|_{H^{s(p_1)}} + \sum_{p\in\{p_0,p_1\}} \|u\|_{X_p}^p\bigl[\|u\|_Y + \||\nabla|^{s(p_1)}u\|_Y\bigr] \\
& \lesssim \|u_-\|_{H^{s(p_1)}} + [\eta^{p_0}+\eta^{p_1}] \|u_-\|_{H^{s(p_1)}}  \leq 4C\|u_-\|_{H^{s(p_1)}} 
\end{align*}
for $\eta$ sufficiently small.   It follows that $\Phi:Z\to Z$.

Next we prove that $\Phi$ is a contraction. Using the properties of $F$ and estimating as above, we have that for $u,v\in Z$, 
\begin{align*}
\|u-v\|_Y & \lesssim \|F(t,x,u)-F(t,x,v)\|_{Y'} \\
& \lesssim \sum_{p\in\{p_0,p_1\}}\| (|u|^p + |v|^p)(u-v)\|_{Y'} \\
& \lesssim \sum_{p\in\{p_0,p_1\}}\biggl[\|u\|_{X_p}^p + \|v\|_{X_p}^p\biggr]\|u-v\|_Y \\
& \lesssim [\eta^{p_0}+\eta^{p_1}]\|u-v\|_Y \leq \tfrac12\|u-v\|_Y
\end{align*}
for $\eta$ sufficiently small. 

It follows that $\Phi$ has a unique fixed point $u\in Z$, which yields the desired solution satisfying $e^{-it\Delta}u(t)\to u_-$ as $t\to-\infty$.

To prove scattering forward in time, we repeat the estimates above to show that $\{e^{-it\Delta}u(t)\}$ is Cauchy in $H^{s(p_1)}$ as $t\to\infty$.  Indeed, writing $Y'(s,t)$ for the space-time norm over $(s,t)\times\R^d$ (and similarly for the other norms), we have
\begin{align*}
\|e^{-it\Delta}u(t)-e^{-is\Delta}u(s)\|_{H^{s(p_1)}} & \lesssim \|F(t,x,u)\|_{Y'(s,t)}+\| |\nabla|^{s(p_1)}F(t,x,u)\|_{Y'(s,t)} \\
& \lesssim [\eta^{p_0}+\eta^{p_1}][\|u\|_{Y(s,t)} + \||\nabla|^{s(p_1)}u\|_{Y(s,t)}]  \to 0  
\end{align*}
as $s,t\to\infty$. Letting $u_+$ denote the limit of $e^{-it\Delta}u(t)$ in $H^{s(p_1)}$, we obtain the last claim in the theorem.  In particular, we obtain the identity \eqref{implicit-formula} for the scattering map:
\begin{equation*}
S_F(u_-)(x)= u_+(x)= u_-(x) - i \int_{-\infty}^\infty e^{-it\Delta} F(t,x,u(t,x))\, dt. \qedhere
\end{equation*}
\end{proof}

\section{Reduction to an inverse convolution problem}\label{S:Reduction}

The goal of this section is to prove Proposition~\ref{P:agree} and Corollary~\ref{C:conv0} below.  These results reduce the proof of Theorem~\ref{T} to an inverse convolution problem.

We recall the notation
$$
G(t,x,|u|^2)=\bar u F(t,x,u) .
$$
We will call $G$ the \emph{potential} associated to $F$.  This is not equal to the potential energy density, but does have the same dimensionality. 

In the next lemma, we show that we may approximate the full solution $u(t)$ in the implicit formula \eqref{implicit-formula} by its first Picard iterate, namely, $e^{it\Delta}u_-$, up to acceptable errors.  The precise statement is the following: 

\begin{lemma}[Born approximation]\label{L:Born} Let $F$ be admissible with parameters $(p_0,p_1)$, potential $G$, and scattering map $S:B_\eta\to H^{s(p_1)}$. Then for any $\varphi\in B_\eta$, 
\[
i\langle (S-I)\varphi,\varphi\rangle = \iint G(t,x,|e^{it\Delta}\varphi|^2)\,dx\,dt + \mathcal{O}\biggl[ \sum_{p\in\{p_0,p_1\}}\|\varphi\|_{H^{s(p)}}^{2p}\|\varphi\|_{L^2}^2\biggr].
\] 
\end{lemma}

\begin{proof} By the Duhamel formula,
\[
i\langle (S-I)\varphi,\varphi\rangle = \int\langle F(t,x,u),e^{it\Delta}\varphi\rangle \,dt,
\]
where $u$ is the solution to \eqref{nls} with $e^{-it\Delta}u(t)\to\varphi$ as $t\to-\infty$. Thus it suffices to prove that
\begin{equation}\label{Born-NTS}
\int\langle F(t,x,u)-F(t,x,e^{it\Delta}\varphi),e^{it\Delta}\varphi\rangle\,dt = \mathcal{O}\biggl[\sum_{p\in\{p_0,p_1\}}\|\varphi\|_{H^{s(p)}}^{2p}\|\varphi\|_{L^2}^2\biggr].
\end{equation}

To prove this, we first introduce
\[
N(t):=u(t)-e^{it\Delta}\varphi=-i\int_{-\infty}^t e^{i(t-s)\Delta} F(s,x,u(s))\,ds 
\]
and notice that
\[
|F(t,x,u)-F(t,x,e^{it\Delta}\varphi)|\lesssim\sum_{p\in\{p_0,p_1\}}\bigl[ |u|^p + |e^{it\Delta}\varphi|^p\bigr] |N(t)|
\]
uniformly in $(t,x)$. Using Strichartz estimates, we obtain 
\begin{align*}
\|N\|_Y &\lesssim \|F(t,x,u)\|_{Y'}  \lesssim \sum_{p\in\{p_0,p_1\}}\|u\|_{X_p}^p \|u\|_{Y}  \lesssim \sum_{p\in\{p_0,p_1\}}\|\varphi\|_{H^{s(p)}}^p \|\varphi\|_{L^2}. 
\end{align*}
Thus
\begin{align*}
|\text{LHS}\eqref{Born-NTS}| & \lesssim \|e^{it\Delta}\varphi\|_Y \biggl\| \sum_{p\in\{p_0,p_1\}}[|u|^p + |e^{it\Delta}\varphi|^p]\, N \biggr\|_{Y'} \\
& \lesssim \|\varphi\|_{L^2} \sum_{p\in\{p_0,p_1\}}\bigl[\|u\|_{X_p}^p + \|e^{it\Delta}\varphi\|_{X_p}^p\bigr]\|N\|_Y \\
& \lesssim \sum_{p\in\{p_0,p_1\}}\|\varphi\|_{H^{s(p)}}^p\|\varphi\|_{L^2} \|N\|_Y \\
& \lesssim \sum_{p\in\{p_0,p_1\}}\|\varphi\|_{H^{s(p)}}^{2p}\|\varphi\|_{L^2}^2,
\end{align*}
as desired. 
\end{proof}

Next, we wish to localize the potential to a fixed point in space-time. 

\begin{lemma}[Space-time localization]\label{L:localize} Let $F$ be admissible, with potential $G$. Fix $(t_0,x_0)\in\R^d$ and $\psi\in\mathcal{S}(\R^d)$. Let
\[
v(t,x)=[e^{it\Delta}\psi](x)\qtq{and}v_\sigma(t,x)=\bigl(e^{it\Delta}[\psi(\tfrac{\cdot}{\sigma})]\bigr)(x)=v(\tfrac{t}{\sigma^2},\tfrac{x}{\sigma}).
\]
Then
\begin{align*}
\iint & G(t,x,|v_\sigma(t-t_0,x-x_0))|^2\,dx\,dt  \\
& = \sigma^{d+2} \iint G(t_0,x_0, |v(t,x)|^2)\,dx\,dt + o_\psi(\sigma^{d+2})\qtq{as}\sigma\to0.
\end{align*}
\end{lemma}

\begin{proof} Introducing the function $H_\sigma$ via 
\begin{align*}
G(t_0+\sigma^2 t,x_0+\sigma x,\lambda ) = H_\sigma(t,x,\lambda )
\end{align*}
and making a change of variables in the integral shows
\begin{align*}
\iint G(t,x,|v_\sigma(t-t_0,x-x_0)|^2)\,dx\,dt 
& = \sigma^{d+2}\iint H_\sigma(t,x,|v(t,x)|^2)\,dx\,dt .
\end{align*}
The proof then reduces to showing that, as $\sigma\to0$,
\begin{equation}\label{H00}
\iint H_\sigma(t,x,|v(t,x)|^2)\,dx\,dt = \iint G(t_0,x_0,|v(t,x)|^2)\,dx\,dt + o_\psi(1).
\end{equation}

To this end, we first note that by the continuity of $F$, we have that
\[
\lim_{\sigma\to 0}H_\sigma(t,x,|v(t,x)|^2)=G(t_0,x_0,|v(t,x)|^2)\qtq{for all}(t,x).
\]
Next, we observe that
\[
|H_\sigma(t,x,|v(t,x)|^2)| \lesssim \sum_{p\in\{p_0,p_1\}} |v(t,x)|^{2p+2} 
\]
uniformly in $(t,x)$ and $\sigma>0$, and we claim that
\[
\psi\in \mathcal{S}(\R^d) \implies \sum_{p\in\{p_0,p_1\}} |v(t,x)|^{2p+2}\in L_{t,x}^1(\R\times\R^d). 
\]
To verify this claim, we observe that by Sobolev embedding and Strichartz estimates, 
\[
\|v\|_{L_{t,x}^{2p+2}} \lesssim \||\nabla|^{\frac{dp-2}{2p+2}}v\|_{L_t^{2p+2} L_x^{\frac{2d(p+1)}{d(p+1)-2}}}\lesssim \||\nabla|^{\frac{dp-2}{2p+2}} \psi\|_{L^2}<\infty 
\]
for $p\in\{p_0,p_1\}$.  

The assertion \eqref{H00} now follows from the dominated convergence theorem.\end{proof}

Combining Lemma~\ref{L:Born} and Lemma~\ref{L:localize}, we obtain the following: 

\begin{proposition}[Pointwise agreement of potentials]\label{P:agree} Let $F_1$ and $F_2$ be admissible nonlinearities, with corresponding potentials $G_1,G_2$.  Denote the corresponding scattering maps by $S_1,S_2$.

Suppose that $S_1=S_2$ on their common domain. Then for any $(t_0,x_0)$, $A>0$, and $\psi\in\mathcal{S}(\R^d)$,  
\[
\iint G_1(t_0,x_0,A|e^{it\Delta}\psi|^2)\,dx\,dt = \iint G_2(t_0,x_0,A|e^{it\Delta}\psi|^2)\,dx\,dt. 
\]
\end{proposition}

\begin{proof} Fix $(t_0,x_0)\in\R\times\R^d$, $A>0$, and $\psi\in\mathcal{S}(\R^d)$, and denote the parameters of $F_1,F_2$ by $(p_0,p_1)$, $(p_0,p_2)$, respectively.  

Given $0<\sigma<1$, we define
\[
\varphi_\sigma(x) =\bigl[e^{-i(t_0/\sigma^2)\Delta}\psi\bigr](\tfrac{x-x_0}{\sigma}).
\]

Noting that
\[
\| \varphi_\sigma \|_{H^s(\R^d)} \lesssim_\psi \sigma^{\frac{d}{2}-s}\qtq{for}s\geq 0
\]
and that $s(p_j)<\tfrac{d}{2}$ for $j\in\{0,1,2\}$, it follows that $\sqrt{A}\varphi_\sigma$ belongs to the common domain of $S_1$ and $S_2$ for $\sigma$ sufficiently small. 

Assuming that $S_1$ and $S_2$ agree on their common domain, the Born approximation (Lemma~\ref{L:Born}) implies
\begin{align*}
\iint [G_2(t,x&, A|e^{it\Delta}\varphi_\sigma|^2(x))-G_1(t,x,A|e^{it\Delta}\varphi_\sigma|^2(x))]\,dx\,dt \\
& = \mathcal{O}_A\biggl[\sum_{p\in\{p_0,p_1,p_2\}}\|\varphi_{\sigma}\|_{H^{s(p)}}^{2p}\|\varphi_\sigma\|_{L^2} ^2\biggr] =\mathcal{O}_A(\sigma^{d+4})
\end{align*}
for small $\sigma$.  Now observe that
\[
[e^{it\Delta}\varphi_\sigma](x) = [e^{i\sigma^{-2}(t-t_0)\Delta}\psi](\tfrac{x-x_0}{\sigma}) =: v_\sigma(t-t_0,x-x_0),
\]
 where we write
\[
v_\sigma(t,x)=\bigl(e^{it\Delta}[\psi(\tfrac{\cdot}{\sigma})]\bigr)(x) = v(\tfrac{t}{\sigma^2},\tfrac{x}{\sigma}),\qtq{with} v(t,x)= e^{it\Delta}\psi.
\]

Thus, applying space-time localization (Lemma~\ref{L:localize}), we obtain
\begin{align*}
\iint & [G_2(t_0,x_0,A|v(t,x)|^2)-G_1(t_0,x_0,A|v(t,x)|^2)]\,dx\,dt  =\mathcal{O}(\sigma^2)+o(1)
\end{align*}
as $\sigma\to 0$.  As the left-hand side does not depend on $\sigma$, we now obtain the desired equality by taking the limit as $\sigma\to 0$. \end{proof}

Next, we rewrite the result of Proposition~\ref{P:agree} so as to exhibit a hidden convolution structure.  We also take this opportunity to specialize to the case of Gaussian data, which is all that we shall need to consider in the next section.  The specific form of the identity appearing in \eqref{conv0-2} is related to the fact that neither factor in the convolution \eqref{conv0-1} is individually well-behaved, and anticipates the analysis of the following section. 

\begin{corollary}[Equality of convolutions]\label{C:conv0} Let $F_1,F_2$ be admissible nonlinearities, with corresponding potentials $G_1,G_2$ and scattering maps $S_1, S_2$. Suppose that $S_1=S_2$ on their common domain. 

Fix $(t_0,x_0)\in\R\times\R^d$ and let $\psi(x)=\exp\{-\tfrac{|x|^2}{4}\}$. Define
\[
g_j(|u|^2) = G_j(t_0,x_0,|u|^2)\qtq{and}h_j(k)= e^{-k}g_j'(e^{-k})\qtq{for}j\in\{1,2\}. 
\]

Then
\begin{equation}\label{conv0-1}
\int_{-a}^\infty [h_1(k)-h_2(k)] \mu(e^{-[k+a]})\,dk = 0 \qtq{for all}a\in\R,
\end{equation}
where $\mu$ is the distribution function
\begin{equation}\label{mu}
\mu(\lambda)  :=  \bigl|\bigl\{(t,x)\in\R\times\R^d:|e^{it\Delta}\psi(x)|^2>\lambda\bigr\}\bigr|.
\end{equation}

In particular, for any $a\in\R$ and any $c\in\R$,
\begin{equation}\label{conv0-2}
\int_0^\infty e^{c(k+\ell)}[h_1(k-a+\ell)-h_2(k-a+\ell)]e^{-ck}\mu(e^{-k})\,dk = 0 \qtq{for all} \ell\geq 0. 
\end{equation}

\end{corollary}

\begin{proof} Fix $a\in\R$. By the Fundamental Theorem of Calculus and the change of variables $\lambda=e^{-k}$, we may obtain 
\begin{align*}
\iint g_j(e^{a}|e^{it\Delta}\psi|^2)\,dx\,dt &  = \int_0^{e^{a}} g_j'(\lambda) \mu(\lambda e^{-a})\,d\lambda \\
& = \int_{-a}^\infty e^{-k}g_j'(e^{-k})\mu(e^{-[k+a]})\,dk \\
& = \int_{-a}^\infty h_j(k)\mu(e^{-[k+a]})\,dk
\end{align*}
for $j\in\{1,2\}$.  Here we have used the fact that $|e^{it\Delta}\psi|^2\leq 1$ for all $(t,x)\in\R\times\R^d$ (cf. \eqref{gaussian-solution} below) and so $\mu(\lambda)=0$ for $\lambda\geq 1$.  This support property of $\mu$ together with the Strichartz inequality and the admissibility of $F_j$ guarantee the absolute convergence of all these integrals. The identity \eqref{conv0-1} now follows from Proposition~\ref{P:agree}. 

To obtain \eqref{conv0-2} from \eqref{conv0-1}, we first change variables in the integral via $k\mapsto k-a$ and then make the replacement $a\mapsto a- \ell$.  Notice that these changes have brought \eqref{conv0-1} into a form more closely resembling \eqref{BL question}
\end{proof}

\section{Deconvolution}\label{S:Deconvolution}

In Corollary~\ref{C:conv0}, we reduced the proof of Theorem~\ref{T} to a deconvolution problem involving the distribution function of a Gaussian solution to the linear Schr\"odinger equation.  Specifically, it remains to show that \eqref{conv0-2} implies $h_1\equiv h_2$, since this in turn yields $F_1\equiv F_2$.  In this section, we resolve the deconvolution problem using Corollary~\ref{C:BL}.  The first step is to compute the Laplace transform of $\mu(e^{-k})$.  Note that this function grows as $k\to \infty$; this limits the region of $z$ for which the integral is convergent.  Nevertheless, we find that the Laplace transform admits the meromorphic continuation \eqref{M} to the whole complex plane.

Recall that the Gamma function is defined by
\[
\Gamma(z) = \int_0^\infty e^{-t} t^{z-1}\,dt \quad\text{when $\Re z>0$}. 
\]
It can then be extended meromorphically to $\C$ via the relation $\Gamma(z)=z^{-1}\Gamma(z+1)$.

\begin{proposition}\label{P:Laplace} Let 
\begin{equation}\label{mu defn}
\psi(x) := \exp\{-\tfrac{|x|^2}{4}\} \qtq{and} \mu(\lambda)  :=  \bigl|\bigl\{(t,x)\in\R\times\R^d:|e^{it\Delta}\psi(x)|^2>\lambda\bigr\}\bigr|.
\end{equation}
Then the integral 
\begin{equation}\label{M defn}
M(z) :=  \int_0^\infty e^{-kz} \mu(e^{-k})\,dk,
\end{equation}
which is absolutely convergent for $\Re z>1+\tfrac{1}{d}$, is given by
\begin{equation}\label{M}
M(z) =  2^{\frac{d}{2}}\pi^{\frac{d+1}{2}} z^{-\frac{d}{2}-1}\frac{\Gamma(\frac{d}{2}(z-1)-\frac12)}{\Gamma(\frac{d}{2}(z-1))} .
\end{equation}
Moreover, this function $M(z)$ obeys the following bounds: 
\begin{equation}\label{M-bounds}
|M(z)|\approx_{d} (1+|z|^2)^{-\frac{d+3}{4}} \qtq{uniformly for}\Re z\geq1+\tfrac{3}{2d}.
\end{equation}
\end{proposition}

\begin{proof} We begin by using the change of variables $\lambda=e^{-k}$ and the fundamental theorem of calculus to write
\begin{align*}
\int_0^\infty e^{-kz}\mu(e^{-k})\,dk & = \int_0^1 \lambda^{z-1}\mu(\lambda)\,d\lambda  = z^{-1}\iint |e^{it\Delta}\psi(x)|^{2z}\,dx\,dt. 
\end{align*}
We now use the fact that
\begin{equation}\label{gaussian-solution}
\psi(x)=\exp\{-\tfrac{|x|^2}{4}\} \implies e^{it\Delta}\psi(x) = \bigl[\tfrac{1}{1+it}\bigr]^{\frac{d}{2}} \exp\{-\tfrac{|x|^2}{4(1+it)}\}
\end{equation}
(see e.g. \cite[Equation~(2.4)]{Visan}), which yields
\[
|e^{it\Delta}\psi(x)|^2 = (1+t^2)^{-\frac{d}{2}}\exp\{-\tfrac{|x|^2}{2(1+t^2)}\}.
\]

Recalling the Gaussian integral 
\[
\int_{\R^d} \exp\{- w |x|^2\}\,dx = (\tfrac{\pi}{w})^{\frac d2},  \qtq{valid for} \Re w >0,
\]
we obtain
\[
M(z) = z^{-1}\iint |e^{it\Delta}\psi(x)|^{2z}\,dx\,dt = 2^{\frac{d}{2}}\pi^{\frac{d}{2}} z^{-\frac{d}{2}-1}\int_\R (1+t^2)^{-\frac{d}{2}[z-1]}\,dt.
\]
This is clearly absolutely convergent if and only $\Re z>1+\tfrac{1}{d}$.

By making the change of variables $s=(1+t^2)^{-1}$, we obtain the following special case of Euler's Beta integral:
\begin{align*}
\int_\R (1+t^2)^{-c}\,dt &= \int_0^1 s^{c-\frac32}(1-s)^{-\frac12}\,ds =\tfrac{\Gamma(\frac12)\Gamma(c-\frac12)}{\Gamma(c)} = \pi^{\frac12} \tfrac{\Gamma(c-\frac12)}{\Gamma(c)}
\end{align*}
with $\Re c>\tfrac12$. This proves \eqref{M}.

We now turn to the bounds in \eqref{M-bounds}.  The key property of the Gamma function that we will use is the following inequality (see \cite[Theorem A, page 68]{Rademacher}):
\[
\biggl| \frac{\Gamma(z+\alpha)}{\Gamma(z)}\biggr| \leq |z|^\alpha \qtq{for}\alpha\in[0,1]\qtq{and}\Re z\geq\tfrac12(1-\alpha).
\]
This directly implies
\begin{equation}\label{MLB}
|M(z)| \geq 2^{\frac{d}{2}}\pi^{\frac{d+1}{2}} |z|^{-\frac{d}{2}-1}|\tfrac{d}{2}(z-1)-\tfrac12|^{-\frac12}
	\qtq{for}\Re z\geq1+\tfrac{3}{2d}. 
\end{equation}
Using the identity $\Gamma(w)=w^{-1}\Gamma(w+1)$, we also obtain
\begin{equation}\label{MUB}
|M(z)| \leq 2^{\frac{d}{2}}\pi^{\frac{d+1}{2}} |z|^{-\frac{d}{2}-1}|\tfrac{d}{2}(z-1)-\tfrac12|^{-1}|\tfrac{d}{2}(z-1)|^{\frac12}
	\qtq{for}\Re z>1+\tfrac{1}{d}. 
\end{equation}
The unified \eqref{M-bounds} follows from \eqref{MLB}, \eqref{MUB}, and the restriction on~$z$.
\end{proof}


\begin{corollary}\label{C:outer} Suppose $\mu(\lambda)$ is defined by \eqref{mu defn}, $c\geq 1+\tfrac{3}{2d}$, 
\begin{equation}\label{V defn}
v(k):= e^{-ck} \mu(e^{-k}), \qtq{ and } V(z) :=  \int_0^\infty e^{-kz} v(k)\,dk.
\end{equation}
Then $V(z)\in\H^2(\mathbb H)$ and is an outer function.
\end{corollary}

\begin{proof} By the restriction on $c$, the integral defining $V(z)$ is absolutely convergent in the half-plane $\Re z \geq 0$; thus, $V(z)$ is analytic there.  To show that it belongs to $\H^2(\mathbb H)$, we must verify \eqref{Hardy cond}.  Using \eqref{M-bounds} makes this easy:
\[
\sup_{x>0} \int_\R |V(x+iy)|^2\,dy = \sup_{x>0} \int_\R |M(c+x+iy)|^2\,dy
	\lesssim \int_{-\infty}^\infty (1+y^2)^{-\frac{(d+3)}2}\,dy <\infty.
\]

To prove that $V(z)$ is outer, we use Proposition~\ref{P:outer} in concert with the lower bound from \eqref{M-bounds}: Choosing $n\in \mathbb N$ with $n> \frac{d+4}2$, we have
\begin{equation*}
\sup_{x>0} \int_\R \frac{dy}{[(1+x)^2+y^2]^n |V(x+iy)|^2} \lesssim \int_{-\infty}^\infty (1+y^2)^{\frac{(d+3)}2 - n }\,dy <\infty. \qedhere
\end{equation*}
\end{proof}

Finally, we are in a position to complete the proof of Theorem~\ref{T}.

\begin{proof}[Proof of Theorem~\ref{T}] We suppose that $F_1,F_2$ are admissible nonlinearities with parameters $(p_0,p_1)$, $(p_0,p_2)$ and scattering maps $S_1,S_2$.  We suppose further that that $S_1$ and $S_2$ agree on their common domain.

Fix $(t_0,x_0)\in\R\times\R^d$ and define $h_1,h_2$ as in Corollary~\ref{C:conv0}.  By \eqref{conv0-2} from that corollary, we know that for any choice of $a,c\in\R$, the two functions
\begin{equation}\label{v and phi}
v(k):=e^{-ck}\mu(e^{-k})  \qtq{and} \varphi_a(k) := e^{ck}[h_1(k-a)-h_2(k-a)]
\end{equation}
satisfy the convolution equation \eqref{BL question}.  Note that $\varphi_a(k)$ is real-valued.

In order to apply Corollary~\ref{C:BL}, we must ensure that $v,\varphi_a\in L^2([0,\infty))$ and that the Laplace transform $V(z)$ of $v(k)$ is an outer function.  

By admissibility of the nonlinearities and recalling the definition of $h_j$ from Corollary~\ref{C:conv0}, we see that $\varphi_a(k) \in L^2([0,\infty))$ for any $a\in\R$ provided $c<1+\tfrac{2}{d}$.  In particular, we may choose $c=1+\tfrac{3}{2d}$.  For this choice, Corollary~\ref{C:outer} shows that  $V(z)\in \H^2(\mathbb H)$ and that it is outer.  Using the Plancherel identity, this proves $v(k)\in L^2([0,\infty))$.

Applying Corollary~\ref{C:BL}, we find that $\varphi_a(k)\equiv 0$ for all almost every $k\in[0,\infty)$ and all $a\geq 0$.  This implies that the continuous functions $h_1(k)$ and $h_2(k)$ agree for all $k\in\R$.  Recalling the definition of $h_j$ from Corollary~\ref{C:conv0}, we obtain that $F_1(t_0,x_0,u)=F_2(t_0,x_0,u)$ for all $u\in\C$.  Finally, as $(t_0,x_0)$ was arbitrary, we conclude $F_1\equiv F_2$.\end{proof}


\end{document}